\documentclass[10pt]{article}

\usepackage{amsmath,amssymb,amscd,amsthm,makeidx,txfonts,graphicx,url,bm}
\usepackage{xypic}
\usepackage[backref]{hyperref}

\numberwithin{equation}{section}

\let\oldmarginpar\marginpar
\renewcommand\marginpar[1]{\-\oldmarginpar[\raggedleft\footnotesize #1]
{\raggedright\footnotesize #1}}

\makeindex

\xyoption{all}
\input{xypic}

\newtheorem{theorem}{Theorem}[section]
\newtheorem{proposition}[theorem]{Proposition}
\newtheorem{corollary}[theorem]{Corollary}

\newtheorem{lemma}[theorem]{Lemma}

\theoremstyle{remark}
\newtheorem{remark}[theorem]{Remark}

\theoremstyle{definition}

\newcounter{margin}
{\end{itshape}  \bigskip}

\def\beq{\begin{eqnarray}}
\def\eeq{\end{eqnarray}}
\def\bes{\begin{eqnarray*}}
\def\ees{\end{eqnarray*}}

\def\muhat{{\bm \mu}}
\def\xihat{{\bm \xi}}

\def\omhat{{\bm \omega}}

\def\F_qc{{\gamma}}

\def\C{\mathbb{C}}

\def\calQ{{\mathcal{Q}}}

\def\calX{{\mathcal{X}}}
\def\calU{{\mathcal{U}}}

\def\calF{{\mathcal{F}}}

\def\calP{\mathcal{P}}
\def\calH{\mathcal{H}}

\def\bG{\mathbb{G}}
\def\DT{{\rm DT}}

\def\bV{\mathbb{V}}

\def\v{\mathbf{v}}
\def\w{\mathbf{w}}
\def\P{\mathcal{P}}

\def\Hs{\mathbb{H}^s}
\def\bfL{\mathbf{L}}
\def\bfP{\mathbf{P}}

\def\bfX{\mathbf{X}}

\def\N{\mathbb{Z}_{\geq 0}}

\def\F{\mathbb{F}}
\def\Q{\mathbb{Q}}

\def\T{\mathbb{T}}

\def\tv{{\tilde{\v}}}

\def\calC{{\mathcal C}}
\def\calO{{\mathcal O}}

\def\Z{\mathbb{Z}}

\def\Gm{\mathbb{G}_m}
\def\K{\mathbb{K}}

\def\gl{{\mathfrak g\mathfrak l}}

\newcommand{\nc}{\newcommand}

\def\IC{{\rm IC}^\bullet_}
\def\pIC{\underline{\rm IC}{^\bullet}_}
\def\bfT{\mathbf{T}}
\def\bfB{\mathbf{B}}

\nc{\op}[1]{\mathop{\mathchoice{\mbox{\rm #1}}{\mbox{\rm #1}}
{\mbox{\rm \scriptsize #1}}{\mbox{\rm \tiny #1}}}\nolimits}
\nc{\al}{\alpha}

\nc{\ep}{\varepsilon} \nc{\ga}{\gamma} \nc{\Ga}{\Gamma}
\nc{\la}{\lambda} \nc{\La}{\Lambda} \nc{\si}{\sigma}
\nc{\Sig}{{\Gamma}} \nc{\Om}{\Omega} \nc{\om}{\omega}

\nc{\SL}{{\rm SL}} \nc{\GL}{{\rm GL}} \nc{\PGL}{{\rm PGL}}

\nc{\cpt}{{\op{cpt}}} \nc{\Dol}{{\op{Dol}}} \nc{\DR}{{\op{DR}}}
\nc{\B}{{\op{B}}} \nc{\Triv}{\op{Triv}} \nc{\Hod}{{\op{Hod}}}
\nc{\Log}{{\op{Log}}} \nc{\Exp}{{\op{Exp}}} \nc{\Est}{E_{\op{st}}}
\nc{\Hst}{H_{\op{st}}} \nc{\Left}[1]{\hbox{$\left#1\vbox to
  10.5pt{}\right.\nulldelimiterspace=0pt \mathsurround=0pt$}}
\nc{\Right}[1]{\hbox{$\left.\vbox to
  10.5pt{}\right#1\nulldelimiterspace=0pt \mathsurround=0pt$}}
\nc{\LEFT}[1]{\hbox{$\left#1\vbox to
  15.5pt{}\right.\nulldelimiterspace=0pt \mathsurround=0pt$}}
\nc{\RIGHT}[1]{\hbox{$\left.\vbox to
  15.5pt{}\right#1\nulldelimiterspace=0pt \mathsurround=0pt$}}

\nc{\bee}{{\bf E}} \nc{\bphi}{{\bf \Phi}}

\begin{document}

\title{DT-invariants of  quivers and the Steinberg character of $\GL_n$}

\author{ Emmanuel Letellier \\ {\it
  Universit\'e de Caen / CNRS UMR 61 39} \\ {\tt
  emmanuel.letellier@unicaen.fr} }

\pagestyle{myheadings}

\maketitle

\begin{abstract} In this paper we give a simple description of DT-invariants of double quivers without potential as the multiplicity of the Steinberg character in some representation associated with the quiver. When the dimension vector is indivisible we use this description to express these DT-invariants as the Poincar\'e polynomial of some singular quiver varieties. Finally we explain the connections  with \cite{aha3} where DT-invariants are expressed as the graded multiplicities of the trivial representation of some Weyl group in the cohomology of some non-singular quiver varieties attached to an extended quiver.

\end{abstract}

\section{Introduction}

Given a finite quiver $\Gamma$ with set of vertices $I=\{1,\dots,r\}$ and a vector dimension $\v=(v_1,\dots,v_r)\in(\N)^r$, we  defined in \cite{aha3}  a family of polynomials $\{\Hs_\muhat(t)\}_\muhat$ indexed by the set $\calP_\v=\calP_{v_1}\times\cdots\times\calP_{v_r}$ of multi-partitions $\muhat=(\mu^1,\dots,\mu^r)$ with $\mu^i$ of size $v_i$. This family contains two well-studied polynomials, namely Kac polynomials $A_\v(q)$ that counts the number of isomorphism classes of absolutely indecomposable representations of $(\Gamma,\v)$ over $\F_q$ and the DT-invariants $\DT_\v(q)$ of the double quiver $\overline{\Gamma}$ without potential. More precisely, if we put $\v^1:=((v_1)^1,\dots,(v_r)^1)\in\calP_\v$ and  $1^\v:=((1^{v_1}),\dots,(1^{v_r}))\in\calP_\v$, we proved that $A_\v(t)=\Hs_{\v^1}(t)$ and $\DT_\v(t)=\Hs_{1^\v}(t)$. The main result of \cite{aha3} is the proof of the positivity of the coefficients of the polynomials $\Hs_\muhat(t)$, which implies Kac conjecture on the positivity of the coefficients of $A_\v(t)$ (see \cite{kac2}) and the positivity for $\DT_\v(t)$. The positivity of $\DT_\v(t)$ was first proved by Efimov \cite{efimov} by a completely different method while the positivity of $A_\v(t)$ was previoulsy proved by Crawley-Boevey and van den Bergh \cite{crawley-etal} for indivisible $\v$ and by Mozgovoy \cite{mozgovoy} when the quiver $\Gamma$ supports at least one loop at each vertex.

To prove the positivity of $\Hs_\muhat(t)$, we defined a certain non-singular quiver variety $\calQ_{\tilde{\v}}$ associated to a certain extended quiver $\tilde{\Gamma}$ and an extended vector dimension $\tv$ of $\tilde{\Gamma}$ which is indivisible. Denote by $W_\v$ the subgroup of the Weyl group of $\tilde{\Gamma}$ generated by the reflexions at the extra vertices $\tilde{I}\backslash I$. It is isomorphic  to $\frak{G}_\v:=\frak{G}_{v_1}\times\cdots\times\frak{G}_{v_r}$, where $\frak{G}_n$ denotes the symmetric group in $n$ letters. Following Nakajima's construction of actions of Weyl group on cohomology of quiver varieties, the group $W_\v=\frak{G}_\v$ acts on the compactly supported cohomology $H_c^i(\calQ_\tv,\C)$. In \cite{aha3}, we prove that the polynomials $t^{\frac{1}{2}{\rm dim}\calQ_\tv}\Hs_\muhat(t)$ are the graded multiplicities of the irreducible character $\chi^{\muhat'}$ of $\frak{G}_\v$, where $\muhat'$ denotes the dual partition of $\muhat$, in the graded $\C[\frak{G}_\v]$-module $H_c^{2\bullet}(\calQ_\tv,\C)$.

Consider now the $\F_q$-vector space ${\rm Rep}_{\F_q}(\Gamma,\v)$ of  representations of $(\Gamma,\v)$ over $\F_q$. The usual action of $\GL_\v(\F_q)=\GL_{v_1}(\F_q)\times\cdots\times\GL_{v_r}(\F_q)$ on this vector space defines a permutation representation $R_\v$ of $\GL_\v(\F_q)$ on the $\C$-vector space $\C[{\rm Rep}_{\F_q}(\Gamma,\v)]$. It is natural to ask about the decomposition of $R_\v$ into irreducibles. Among the irreducible representations of $\GL_\v(\F_q)$ we have the so-called  \emph{unipotent} representations $\{\calU_\muhat\}_\muhat$ that are parametrized by the set $\calP_\v$. In this parametrization,  the trivial representation corresponds to $\v^1$ and the Steinberg representation ${\rm St}_\v$ to $1^\v$ (see \S \ref{unipotent}). Choose a linear character $\alpha_\v:\Gm(\F_q)\rightarrow\C^\times$ of order $\sum_iv_i$ and define the linear character $1^{\alpha_\v}$ of $\GL_\v(\F_q)$ as 

$$
1^{\alpha_\v}(g)=\prod_i\alpha_\v({\rm det}(g_i)),
$$
where $g=(g_1,\dots,g_r)\in\GL_\v$.

Then one can prove using similar calculations as in \cite[Section 6.10]{letellier2} that the polynomials $\Hs_\muhat(q)$ defined in \cite{aha3} coincides with the multiplicity of the irreducible representation $1^{\alpha_\v}\otimes\calU_\muhat$ in $R_\v$, proving thus that the Kac polynomial $A_\v(q)$ is the multiplicity of $1^{\alpha_\v}$ in $R_\v$ and the DT-invariant $\DT_\v(q)$ is the multiplicity of $1^{\alpha_\v}\otimes {\rm St_\v}$.

The connection just mentioned between the Steinberg character and the DT-invariants, which has not been written in the literature, is rather complicated. In the first part of this paper  we give a simple proof of the connection between the two. 

In the second part of the paper, we assume that $\v$ is indivisible and  we give an alternative description of the polynomials $\Hs_\muhat(q)=\left\langle R_\v,1^{\alpha_\v}\otimes{\rm St}_\v\right\rangle_{\GL_\v}$ as the Poincar\'e polynomial (for intersection cohomology) of some quiver varieties attached to adjoint orbits of $\gl_\v$. More precisely, we choose a generic parameter $\xihat=(\xi_1,\dots,\xi_r)\in\C^r$ (which is possible because $\v$ is indivisible) and we denote by $C_\muhat$ the adjoint orbit $\xihat+\calO_\muhat$ of $\gl_\v(\C)$ where  $\calO_\muhat$ is the nilpotent orbit of $\gl_\v(\C)$ which Jordan blocks of size given by $\muhat$. We define our quiver variety $\calQ_{\overline{C}_\muhat}$ as the affine GIT quotient of $\mu_\v^{-1}(\overline{C}_\muhat)$ by $\GL_\v$ where $\mu_\v$ is the moment map defined in \S \ref{quiver}.

Here is the main result of this paper.

\begin{theorem} We have

$$
A_\v(q)=\left\langle R_\v,1^{\alpha_\v}\right\rangle,\hspace{.5cm}\DT_\v(q)=\left\langle R_\v,1^{\alpha_\v}\otimes{\rm St}_\v\right\rangle.
$$

Assuming that $\v$ is indivisible, we have

\beq
\left\langle R_\v,1^{\alpha_\v}\otimes\calU_\muhat\right\rangle_{\GL_\v}=q^{-\frac{1}{2}{\rm dim}\,\calQ_{\overline{C}_{\muhat'}}}\sum_i{\rm dim}\, IH_c^{2i}(\calQ_{\overline{C}_{\muhat'}},\C)\, q^i,
\label{theointro}\eeq
where $IH_c^i(\calQ_{\overline{C}_\muhat},\C)$ is the compactly supported intersection cohomology of $\calQ_{\overline{C}_\muhat}$. As a consequence,

\begin{align}&A_\v(q)=q^{-\frac{1}{2}{\rm dim}\,\calQ_\xihat}\sum_i{\rm dim}\, H_c^{2i}(\calQ_\xihat,\C)\, q^i,\\
&\DT_\v(q)=q^{-\frac{1}{2}{\rm dim}\,\calQ_{\overline{C}_{\v^1}}}\sum_i{\rm dim}\, H_c^{2i}(\calQ_{\overline{C}_{\v^1}},\C)\, q^i.
\end{align}
\end{theorem}

Formula (1.3) follows from (\ref{theointro}) because the variety $\calQ_{\overline{C}_{\v^1}}$ is \emph{rationally smooth} (see the end of the proof of Theorem \ref{Theta}). The formula $A_\v(q)=q^{-\frac{1}{2}{\rm dim}\,\calQ_\xihat}\sum_i{\rm dim}\, H_c^{2i}(\calQ_\xihat,\C)\, q^i$ is a theorem of Crawley-Boevey and van den Bergh  \cite{crawley-etal}. Their proof is completely different from the one here. \bigskip

Let $S$ be a generic semisimple regular adjoint orbit of $\gl_\v(\C)$ (see \S \ref{quiver}). Such an orbit $S$ always exists (even for divisible dimension vector) and the associated quiver variety $\calQ_S:=\mu_\v^{-1}(S)//\GL_\v$, which was denoted by $\calQ_\tv$ in \cite{letellier2} is non-singular. Under the assumption that $\v$ is indivisible, it is possible to construct an algebraic action of $\frak{G}_\v$ on $H_c^i(\calQ_S,\C)$ using the Springer resolution of $\overline{C}_{\v^1}$ as explained above Corollary \ref{gradedres}. This construction of the action works also in positive characteristics and is compatible with base change (where in positive characteristic we replace the usual cohomology with $\ell$-adic cohomology). The results of this paper shows that $q^{\frac{1}{2}{\rm dim}\,\calQ_S}\langle R_\v,1^{\alpha_\v}\otimes\calU_\muhat\rangle$ is the graded multiplicity of $\chi^{\muhat'}$ in the $\C[\frak{G}_\v]$-module $H_c^{2\bullet}(\calQ_S,\C)$ (see Corollary \ref{gradedres}). In \cite{aha3}, we do not assume that $\v$ is indivisible and so the above construction of the action of $\frak{G}_\v$ does not work as $C_{\v^1}$ does not exist. Instead we use an alternative construction due to Nakajima \cite{nakajima2} which has the advantage to be able to be defined without the Springer resolution but which, unlike the Springer construction, is complicated to define in positive characteristic due to a technical problem (the positive characteristic case is needed as we count points over finite fields). To overtake this technical problem in \cite{aha3} we used the construction in characteristic zero that we transfer to positive characteristic by base change technics, and  we proved that $q^{\frac{1}{2}{\rm dim}\,\calQ_S}\Hs_\muhat(q)=q^{\frac{1}{2}{\rm dim}\,\calQ_S}\langle R_\v,1^{\alpha_\v}\otimes\calU_\muhat\rangle$ is the graded multiplicity (for Nakajima's action) of $\chi^{\muhat'}$ in the graded $\C[\frak{G}_\v]$-module $H_c^{2\bullet}(\calQ_S,\C)$. The results of this paper together with those of \cite{aha3} show that in the indivisible case, Nakajima's action of $\frak{G}_\v$ on $H_c^i(\calQ_S,\C)$ coincides with the Springer action as their decomposition into irreducible representations of $\frak{G}_\v$ is the same.

\section{DT-invariants and Steinberg characters}

We fix once for all a finite quiver $\Gamma$ with set of vertices $I=\{1,\dots,r\}$ and set of arrows $\Omega$. Let $\K$ be a field. For a dimension vector $\v=(v_i)_{i\in I}\in (\N)^I$, we put 

$$
{\rm Rep}_\K(\Gamma,\v):=\bigoplus_{i\rightarrow j\in\Omega}{\rm Mat}_{v_j,v_i}(\K),\hspace{.5cm} \GL_\v:=\prod_{i\in I}\GL_{v_i}(\K),\hspace{.5cm}\gl_\v=\bigoplus_{i\in I}\gl_{v_i}(\K),$$
and we consider the usual action of $\GL_\v$ on ${\rm Rep}_\K(\Gamma,\v)$. We will also denote by $\overline{\Gamma}$ the so-called \emph{double quiver} of $\Gamma$, namely $\overline{\Gamma}$ has the same vertices as $\Gamma$, but for each arrow $\gamma\in\Omega$ going from $i$ to $j$, we add a new arrow $\gamma^*$ going from $j$ to $i$. We denote by $\overline{\Omega}=\Omega\coprod\Omega^{\rm opp}$ the set of arrows of $\overline{\Gamma}$. 

\subsection{Log compatible class functions}

\subsubsection{Definition of Log}

Denote by $\Lambda[[T]]$ the ring $\Q(t)\otimes_\Z\Z[[T_1,\dots,T_r]]$ of formal series in the commuting variables $T_1,\dots,T_r$ with coefficients in the field $\Q(t)$ of rational functions in $t$. Denote $\Lambda[[T]]^+$ the maximal ideal of series with zero constant term.

For each integer $d>0$, we have the Adams operation $\psi_d:\Lambda[[T]]^+\rightarrow\Lambda[[T]]^+$ $$f(t;T_1,\dots,T_r)\mapsto f(t^d;T_1^d,\dots,T_r^d).$$Consider $\Psi:\Lambda[[T]]^+\rightarrow\Lambda[[T]]^+$ 

$$
\Psi(f)=\sum_{d\geq 1}\frac{\psi_d(f)}{d},
$$
and its inverse $\Psi^{-1}$ given by 
$$
\Psi^{-1}(f)=\sum_{d\geq 1}\mu(d)\frac{\psi_d(f)}{d},
$$
where $\mu$ is the ordinary M\"obius function. We then define $\Log:1+\Lambda[[T]]^+\rightarrow\Lambda[[T]]^+$ and its inverse $\Exp:\Lambda[[T]]^+\rightarrow 1+\Lambda[[T]]^+$ as 
$$
\Log(f)=\Psi^{-1}(\log(f))\hspace{1cm}\text{and}\hspace{1cm}\Exp(f)=\exp(\Psi(f)).
$$
We denote by $\calP$ the set of all partitions including the zero partition which we denote by $0$. For a dimension vector $\v=(v_1,\dots,v_r)\in(\N)^r$, we denote by $\P_\v\subset\calP^r$ the set of multi-partitions $(\mu^1,\dots,\mu^r)$ with $\mu^i$ of size $|\mu^i|=v_i$ for each $i=1,\dots,r$. The size of $\muhat=(\mu^1,\dots,\mu^r)$ is the $r$-tuple $|\muhat|:=(|\mu^1|,\dots,|\mu^r|)$. A \emph{type} is then a function $\omhat:\N\times\calP^r\rightarrow\N$ such that its support $S_\omhat:=\{(d,\muhat)\,|\, \omhat(d,\muhat)\neq 0\}$ is finite and does not contain pairs of the form $(0,\muhat)$ or $(d,0)$. We denote by $\T$ the set of all types. The size $|\omhat|$ of a type $\omhat$ is the $r$-tuple of non-negative integers whose $i$-th coordinate is given by 
$$\sum_{(d,\muhat)\in S_\omhat}d\cdot|\mu^i|\cdot \omhat(d,\muhat).
$$
The subset of types of size $\v=(v_1,\dots,v_r)$ is denoted by $\T_\v$.

For a type $\omhat$ we put 

$$
C_\omhat^o:=\begin{cases}\frac{\mu(d)}{d}(-1)^{r_\omhat-1}\frac{(r_\omhat-1)!}{\prod_{\muhat}\omhat(d,\muhat)!}& \text{ if  there is no } (d',\muhat)\in S_\omhat \text{ with } d'\neq d,\\
0&\text{ otherwise,}\end{cases}
$$
where $r_\omhat:=\sum_{(d,\muhat)}\omhat(d,\muhat)$.

For a family $\{H_\muhat(t)\}_\muhat$ of elements of $\Q(t)$ indexed by multi-partitions $\muhat\in\calP_\muhat$ with $H_0=1$, we extend it to types $\omhat\in\T$ as follows

$$
H_\omhat(t):=\prod_{(d,\muhat)}H_\muhat(t^d)^{\omhat(d,\muhat)}.
$$

For a dimension vector $\v=(v_1,\dots,v_r)$, we put $T^\v:=T_1^{v_1}\cdots T_r^{v_r}$.

We have the following lemma \cite[Formula (2.3.9)]{aha}.

\begin{lemma} For any family $\{H_\muhat(t)\}_\muhat$ as above, we have 

$$
\Log\left(\sum_{\muhat\in\calP^r}H_\muhat(t) T^{|\muhat|}\right)=\sum_{\omhat\in\T\backslash\{0\}}C_\omhat^o H_\omhat(t) T^{|\omhat|}.
$$
\label{Logaha}\end{lemma}

\subsubsection{Log compatible functions}

Let $\F_q$ be a finite field with $q$ elements and let $\overline{\F}_q$ denote an algebraic closure. Denote by $F$ the Frobenius $x\mapsto x^q$ on $\overline{\F}_q$.

Denote by $\mathfrak{O}$ the set of $F$-orbits on $\Gm(\overline{\F}_q)$. Recall that the conjugacy classes of $\GL_n(\F_q)$ are parametrized by the set of all maps $h:\mathfrak{O}\rightarrow\calP$ such that 

$$
\sum_{\gamma\in\mathfrak{O}}|\gamma|\cdot|h(\gamma)|=n.
$$
The type of a conjugacy class of $\GL_n(\F_q)$ corresponding to $h:\mathfrak{O}\rightarrow\calP$ is the map $\omega:\N\times\calP\rightarrow\N$ that assigns to $(d,\mu)$ the number of Frobenius orbit $\gamma\in\mathfrak{O}$ of size $d$ such that $h(\gamma)=\mu$.

Given a dimension vector $\v=(v_1,\dots,v_r)$, the conjugacy classes of $\GL_\v(\F_q)$ as well as their types are defined in a similar manner. Namely the conjugacy classes of $\GL_\v(\F_q)$ are now parametrized by maps $h=(h_1,\dots,h_r):\mathfrak{O}\rightarrow\calP^r$ such that $\sum_{\gamma\in\mathfrak{O}}|\gamma|\cdot |h_i(\gamma)|=v_i$ for all $i=1,\dots,r$. The type of such an $h$ is the map $\omhat\in\T$ that assigns to $(d,\muhat)\in\N\times\calP^r$, the number of $\gamma\in\mathfrak{O}$ of degree $d$ such that $h(\gamma)=\muhat$. 

Assume given a family of functions $F_\v:\GL_\v(\F_q)\rightarrow\Q$ indexed by $\v\in(\N)^r$ (with $F_0=1$) which are constant on conjugacy classes. We further assume that they are constant on conjugacy classes of same type and that their values on types $\omhat$ is of the form $F_\omhat(q)$ where $F_\omhat$ is a rational function in $\Q(t)$. Say that such a family is \emph{Log compatible} if the rational functions $F_\omhat$ satisfy the following property 

$$
F_\omhat(t)=\prod_{(d,\muhat)\in S_\omhat}F_\muhat(t^d)^{\omhat(d,\muhat)},
$$
where $F_\muhat$ is the polynomial giving the values of $F_{|\muhat|}$ at the unipotent conjugacy classes of $\GL_{|\muhat|}(\F_q)$ of type $\muhat\in\calP^r$.

Let $\v=(v_1,\dots,v_r)$ be a non-zero dimension vector and fix an algebraically closed field $\kappa$ of characteristic zero (which for us will be either $\C$ or $\overline{\Q}_\ell$).
Assume given a linear character $\alpha_\v:\Gm(\F_q)\rightarrow\kappa^\times$ of order $\sum_i v_i$. Such a character $\alpha_\v$ exists if and only if $\sum_iv_i$ divides $q-1$. 

Define the inner product of two functions $f,g :\GL_\v(\F_q)\rightarrow\kappa$ constant on conjugacy classes as

$$
\langle f,g\rangle=\langle f,g\rangle_{\GL_\v(\F_q)}:=\frac{1}{|\GL_\v(\F_q)|}\sum_{x\in\GL_\v(\F_q)}f(x)\overline{g(x)}.
$$
We denote by $1$ the identity character of $\GL_\v(\F_q)$ and by $1^{\alpha_\v}$ the twisted linear character given by 

$$
1^{\alpha_\v}(x)=\prod_{i=1}^r(\alpha_\v\circ{\rm det})(x_i),
$$
for $x=(x_1,\dots,x_r)\in\GL_\v(\F_q)$.

We have the following theorem.

\begin{theorem} Assume that the family $\{F_\v\}_\v$ is Log compatible. Then there exist  rational functions $V_\v(t), V_\v^{\rm gen}(t)\in\Q[t]$ such that for any finite field $\F_q$, we have $$V_\v(q)=\langle F_\v,1\rangle,$$
and for any finite field $\F_q$ and any character $\alpha_\v$ of $\Gm(\F_q)$ of order $\sum_iv_i$, we have 

$$
V_\v^{\rm gen}(q)=\langle F_\v,1^{\alpha_\v}\rangle.
$$
Moroever $V_\v$ and $V_\v^{\rm gen}$ are related as follows :
\begin{align*}
(t-1)\,\Log\left(\sum_{\muhat\in\calP^r}\frac{1}{Z_\muhat(t)}F_\muhat(t)\, T^{|\muhat|}\right)&=\Log\left(\sum_{\v\in(\N)^r}V_\v(t)\, T^\v\right)\\
&=\sum_{\v\in(\N)^r\backslash\{0\}}V_\v^{\rm gen}(t)\, T^\v,
\end{align*}
where $Z_\muhat(t)$ is the polynomial whose evaluation at $q$ is the cardinality of the centralizer $C_{\GL_\v(\F_q)}(u)$  of a unipotent element $u\in\GL_\v$ of type $\muhat$.
\label{Logcomp}\end{theorem}

\begin{proof} The existence of the rational function $V_\v(t)$ is clear from the assumptions on $F_\v$. Let us  prove the existence of $V_\v^{\rm gen}$ and that

$$
(t-1)\,\Log\left(\sum_{\muhat\in\calP^r}\frac{1}{Z_\muhat(t)}F_\muhat(t)\, T^{|\muhat|}\right)=\sum_{\v\in(\N)^r\backslash\{0\}}V_\v^{\rm gen}(t)\, T^\v.
$$

Fix $\v=(v_1,\dots,v_r)\neq 0$. Since $\GL_\v(\F_q)\rightarrow \N$, $g\mapsto |C_{\GL_\v(\F_q)}(g)|$ is Log compatible we have 

$$
\left\langle F_\v,1^{\alpha_\v}\right\rangle=\sum_{\omhat\in\T_\v}\frac{1}{Z_\omhat(q)}F_\omhat(q)\sum_C 1^{\alpha_\v}(C),
$$
where the last sum is over the set of conjugacy classes of $\GL_\v(\F_q)$ of type $\omhat$.

We are reduced to see that 

\beq
\sum_C 1^{\alpha_\v}(C)=(q-1)C_\omhat^o,
\label{om}\eeq
where the sum is over conjugacy classes of a given type $\omhat$. Indeed we get what we want by setting 

$$
V_\v^{\rm gen}(t)=\sum_{\omhat\in\T_\v}C_\omhat^o\frac{1}{Z_\omhat(t)}F_\omhat(t),
$$
and by using Lemma \ref{Logaha}.

The proof of the additive analog of (\ref{om}) is contained in the proof of \cite[Proposition 2.10]{aha3}. For the convenience of the reader we sketch it here.
We embed $\GL_\v$ in $\GL_N$ with $N=\sum_iv_i$. For $\omhat\in\T_\v$, we put $\omega:\N\times\calP\rightarrow\N$ with $\omega(d,\mu)=\sum_{\muhat,\,\mu=\cup\muhat}\omhat(d,\muhat)$ where $\cup\muhat=\mu^1\cup\mu^2\cup\cdots\cup\mu^r$ for any $\muhat\in\calP^r$. If $C$ is a conjugacy class of $\GL_\v(\F_q)$ of type $\omhat$, then the $\GL_N(\F_q)$-conjugacy class of $C$ is of type $\omega\in\T_N$. Choose an element of $\GL_\v(\F_q)$ of type $\omhat$ with Jordan form $\sigma u$ where $\sigma$ is semisimple and $u$ is unipotent. Let $L$ be $C_{\GL_N}(\sigma)$. Note that $L$ may not be in $\GL_\v$ but its center $Z_L$ does belong to $\GL_\v$. We denote by $(Z_L)_{\rm reg}$ the subset of $z\in Z_L$ such that $C_{\GL_N}(z)=L$. The map which sends $z\in(Z_L)_{\rm reg}(\F_q)$ to the $\GL_\v(\F_q)$-conjugacy class of $zu$ is surjective on the set of conjugacy classes of $\GL_\v(\F_q)$ of type $\omhat$. The fibers of that map can be identified with 

$$
\{g\in \GL_\v(\F_q)\,|\,gLg^{-1}=L,\, gC_ug^{-1}=C_u\}/M,
$$
where $M=C_{\GL_\v}(u)$
and $C_u$ is the $M$-conjugacy class of $u$. This group be further computed and we find that it is isomorphic to 

$$
W(\omhat)=\prod_{(d,\muhat)}\left((\Z/d\Z)^{\omhat(d,\muhat)}\rtimes \frak{G}_{\omhat(d,\muhat)}\right),$$
where $\frak{G}_n$ denotes the symmetric group in $n$ letters.

We thus reduced have

\begin{align*}
\sum_C1^{\alpha_\v}(C)&=\frac{1}{|W(\omhat)|}\sum_{z\in (Z_L)_{\rm reg}(\F_q)}1^{\alpha_\v}(zu)\\
&=\frac{1}{|W(\omhat)|}\sum_{z\in (Z_L)_{\rm reg}(\F_q)}1^{\alpha_\v}(z),
\end{align*}
Now we regard $1^{\alpha_\v}$ as a character of $\GL_N(\F_q)$. Our choice of $\alpha_\v$ ensures that the character $1^{\alpha_\v}$ of $\GL_N(\F_q)$ is \emph{generic}, namely its restriction to the center $Z_{\GL_N}$ is trivial while its restriction to the center of any proper Levi subgroup of $\GL_N$ is non-trivial. Therefore we can use \cite[Proposition 6.8.5]{letellier2} to deduce the identity (\ref{om}).

Let us now prove  the first identity. Recall that $\phi_d(q)=\frac{1}{d}\sum_{r|d}\mu(r)(q^{d/r}-1)$ is the number of elements of $\frak{O}$ of size $d$. Thanks to \cite[Lemma 2.1.2]{aha2} we are reduced to prove the following identity

$$
\sum_\v\left\langle F_\v,1\right\rangle\, T^\v=\prod_{d=1}^\infty\left(\sum_{\muhat\in\calP^r}\frac{1}{Z_\muhat(q^d)}F_\muhat(q^d)T^{d\cdot|\muhat|}\right)^{\phi_d(q)},
$$
where $d\cdot|\muhat|=(d|\mu^1|,\dots,d|\mu^r|)$. The proof of this identity is formal and uses the parametrization of the set of all conjugacy classes of $\GL_\v(\F_q)$, where $\v$ runs over $(\N)^r$, by the set of all finitely supported maps $\frak{O}\rightarrow\calP^r$.
\end{proof}

\subsection{Example: Kac polynomials}\label{Kac}

Fix a non-zero vector dimension $\v=(v_1,\dots,v_r)$. Define $R_\v:\GL_\v(\F_q)\rightarrow\Q$ by 

$$
R_\v(g):=\#\{\rho\in{\rm Rep}_{\F_q}(\Gamma,\v)\,|\, g\cdot\rho=\rho\}.
$$
This is the character of the representation of $\GL_\v(\F_q)$ on the $\Q$-vector space $\Q[{\rm Rep}_{\F_q}(\Gamma,\v)]$ generated by the elements of the finite set ${\rm Rep}_{\F_q}(\Gamma,\v)$.

We have the following theorem due to Hua \cite[Proof of Theorem 4.3]{hua}.

\begin{theorem} The family $\{R_\v\}_\v$ is Log compatible and for any $\muhat\in\calP^r$, we have 

$$
R_\muhat(t)=\prod_{i\rightarrow j\in\Omega}t^{\langle\mu^i,\mu^j\rangle},
$$
where for two partitions $\lambda$ and $\mu$,  

$$
\langle\lambda,\mu\rangle=\sum_{i,j}{\rm min}(i,j)m_i(\lambda)m_j(\mu),
$$
with $m_i(\lambda)$ the multiplicity of the part $i$ in the partition $\lambda$.
\end{theorem}

From Burnside theorem, the inner product $\langle R_\v,1\rangle$ is exactly the total number of isomorphism classes of ${\rm Rep}_{\F_q}(\Gamma,\v)$ and so from \cite{hua} we find that the polynomials $A_\v(t)\in\Q[t]$ defined by 

$$
\Log \left(\sum_\v\langle R_\v,1\rangle\, T^\v\right)=\sum_{\v>0}A_\v(q)\, T^\v
$$
are the so-called \emph{Kac polynomials}, namely, $A_\v(q)$ counts the number of isomorphism classes of the representations in ${\rm Rep}_{\F_q}(\Gamma,\v)$ that are indecomposable after extension of scalars from $\F_q$ to $\overline{\F}_q$. Such representations are called \emph{absolutely indecomposable}.

Hence from Theorem \ref{Logcomp}, we deduce the following proposition.

\begin{proposition}Assume that $\alpha_\v$ is a linear character of $\Gm(\F_q)$ of order $\sum_i v_i$. Then

$$
A_\v(q)=\langle F_\v,1^{\alpha_\v}\rangle.
$$
\label{Kacprop}\end{proposition}

\begin{remark}We can also prove directly this identity using the well-known fact that an element of ${\rm Rep}_{\F_q}(\Gamma,\v)$ is absolutely indecomposable if and only if the quotient of its stabilizer in $\GL_\v$ is isomorphic to $\overline{\F}_q^\times$. Indeed, in view of the genericity assumption on $\alpha_\v$, the restrictions of $\alpha_\v$ to stabilizers of representations that contains non-scalar semisimple elements will be non-trivial and therefore the inner product formula will vanish on these stabilizers. \end{remark}

\subsection{DT-invariants and the Steinberg character}

The DT-invariants of $\overline{\Gamma}$ without potential are defined from cohomological Hall algebra in \cite{kontsevich-soibelman2}. There are defined combinatorially as follows.

For $\v=(v_1,\dots,v_r)\in(\N)^r$, put 

$$
\delta(\v):=\sum_{i=1}^rv_i,\hspace{2cm} \overline{\gamma}(\v)=\sum_{i=1}^rv_i^2-\sum_{i\rightarrow j\in\overline{\Omega}} v_iv_j.
$$
Then the DT-invariants $\DT_\v(q)$ of $\overline{\Gamma}$ without potential are defined by

$$
(q-1)\Log\left(\sum_{\v\in(\N)^r}\frac{(-1)^{\delta(\v)}q^{-\frac{1}{2}(\overline{\gamma}(\v)+\delta(\v))}}{\prod_{i=1}^r(1-q^{-1})\cdots(1-q^{-v_i})}\, T^\v\right)=\sum_{\v\in(\N)^r\backslash\{0\}}(-1)^{\delta(\v)}\,\DT_\v(q)\, T^\v.
$$

We now recall some well-known fact about the so-called Steinberg character ${\rm St}_n:\GL_n(\F_q)\rightarrow\C$ of $\GL_n(\F_q)$. If $s\in\GL_n(\F_q)$ is a semisimple element, then there exist pairs $(d_1,n_1),\dots,(d_m,n_m)$ such that the centralizer $C_{\GL_n(\overline{\F}_q)}(s)$ is isomorphic to  $\prod_{i=1}^m\GL_{n_i}(\overline{\F}_q)^{d_i}$ and $C_{\GL_n(\F_q)}(s)\simeq\prod_{i=1}^m\GL_{n_i}(\F_{q^{d_i}})$. For an algebraic group $H$ defined over $\F_q$, we put $\epsilon_H=(-1)^{\F_q-\text{rank}(H)}$ (see  \cite[Chapter 8]{DM} for the definition of $\F_q$-rank). The $\F_q$-rank  of $C_{\GL_n(\overline{\F}_q)}(s)$ is  $\sum_in_i$. In particular the $\F_q$-rank of $\GL_\v(\F_q)$ is $\delta(\v)$.   

The Steinberg character of $\GL_n(\F_q)$ is then given by the following formula \cite[Corollary 9.3]{DM}

\beq
{\rm St}_n(g)=\begin{cases}\epsilon_{\GL_n}\epsilon_{C_{\GL_n}(g)}|C_{\GL_n(\F_q)}(g)|_p&\text{ if }g \text{ is semisimple},\\0&\text{ otherwise,}\end{cases}
\label{Stfor}\eeq
where for a finite group $H$, we denote by $|H|_p$ the $p$-part of $|H|$. In particular, ${\rm St}_n(1)=q^{\frac{n(n-1)}{2}}$.

The Steinberg character of $\GL_\v(\F_q)$ is ${\rm St}_\v=\boxtimes_{i=1}^r{\rm St}_{v_i}$. From Formula (\ref{Stfor}) we deduce the following proposition.

\begin{proposition} The family $\{(-1)^{\delta(\v)}{\rm St}_\v\}_\v$ is Log compatible.
\end{proposition}

Let $R_\v$ be as in \S \ref{Kac}. The product of two Log compatible functions being Log compatible, we deduce that the family $\{(-1)^{\delta(\v)}{\rm St}_\v\otimes R_\v\}_\v$ is Log compatible and so by Theorem \ref{Logcomp} we deduce that

$$
(q-1)\Log\left(\sum_{\muhat\in\calP^r}\frac{(-1)^{\delta(|\muhat|)}}{Z_\muhat}R_\muhat(q)\, {\rm St}_\muhat(q)\, T^{|\muhat|}\right)=\sum_{\v>0}(-1)^{\delta(\v)}\langle R_\v,1^{\alpha_\v}\otimes {\rm St}_\v\rangle\, T^\v.
$$
Now by Formula (\ref{Stfor}) we have ${\rm St}_\muhat(q)=0$ unless the coordinates of $\muhat$ are all of the form $(1^{v_i})$ in which case ${\rm St}_\muhat(q)=q^{-\frac{1}{2}(\delta(\v)-\sum_iv_i^2)}$.

Hence

$$
\sum_{\muhat\in\calP^r}\frac{(-1)^{\delta(|\muhat|)}}{Z_\muhat}R_\muhat(q)\, {\rm St}_\muhat(q)=\sum_{\v\in(\N)^r}\frac{(-1)^{\delta(\v)}\,|{\rm Rep}_{\F_q}(\Gamma,\v)| q^{-\frac{1}{2}(\delta(\v)-\sum_iv_i^2)}}{|\GL_\v(\F_q)|}.
$$
Since 

$$
|\GL_\v(\F_q)|=q^{\sum_iv_i^2}\prod_{i=1}^r(1-q^{-1})(1-q^{-2})\cdots(1-q^{-v_i}),\hspace{1cm}  |{\rm Rep}_{\F_q}(\Gamma,\v)|=q^{\sum_{i\rightarrow j\in\Omega}v_iv_j}
$$we deduce that 

$$
\sum_{\muhat\in\calP^r}\frac{(-1)^{\delta(|\muhat|)}}{Z_\muhat}R_\muhat(q)\, {\rm St}_\muhat(q)=\sum_{\v\in(\N)^r}\frac{(-1)^{\delta(\v)}q^{-\frac{1}{2}(\overline{\gamma}(\v)+\delta(\v))}}{\prod_{i=1}^r(1-q^{-1})\cdots(1-q^{-v_i})}
$$
and so that 

\begin{proposition}

$$
\DT_\v(q)=\langle R_\v,1^{\alpha_\v}\otimes {\rm St}_\v\rangle.
$$

\label{DTprop}\end{proposition}

\section{Cohomological interpretation of DT-invariants}

Assume from now that $\v$ is \emph{indivisible}, namely, the gcd of the coordinates of $\v$ equals $1$. Unless specified, the field $\K$ will be either $\C$ or an algebraic closure $\overline{\F}_q$ of a finite field $\F_q$. 

The letter $\ell $ will denote a prime different from the characteristic of $\K$, and the letter $\kappa$ will denote either $\C$ if $\K=\C$ or $\overline{\Q}_\ell$ if $\K=\overline{\F}_q$. For an  algebraic variety $X/_\K$ over $\K$, we denote by $H_c^i(X_\K,\kappa)$ the compactly supported  cohomology (this is $\ell$-adic cohomology if $\K=\overline{\F}_q$ and usual cohomology if $\K=\C$). By work of Deligne, the cohomology $H_c^i(X/_\K,\kappa)$ is endowed with a weight filtration. We say that $H_c^i(X/_\K,\kappa)$ is \emph{pure} if this weight filtration is trivial namely if :

(i) $\K=\C$ and $H_c^i(X/_\K,\kappa)$ has  a pure mixed Hodge structure of weight $i$ for all $i$. 

(ii) $\K=\overline{\F}_q$, $X/_\K$ is defined over $\F_q$ and for all $i$, the eigenvalues of the Frobenius acting on $H_c^i(X/_\K,\kappa)$  have absolute value equals to $q^{i/2}$.

If $X/_\K$ is equidimensional, we denote by $\IC X$ the intersection cohomology complex with coefficients in the constant sheaf  $\kappa$ as defined by Goresky-MacPherson-Deligne.   Say that $X/_\K$ is \emph{rationally smooth} if $\IC X$ is  the complex with $\kappa$ in degree $0$ and $0$ in other degrees. The compactly supported $i$-th intersection cohomology, denoted by $IH_c^i(X/_\K,\kappa)$, is defined as the $i$-th hypercohomology of $\IC X$. In particular, if $X/_\K$ is rationally smooth, we have $IH_c^i(X/_\K,\kappa)=H_c^i(X/_\K,\kappa)$. The cohomology groups $IH_c^i(X/_\K,\kappa)$ are also endowed with a weight filtration (arising from work of Saito if $(\K,\kappa)=(\C,\C)$ and from eigenvalues of Frobenius when $(\K,\kappa)=(\overline{\F}_q,\kappa)$ and $X/_\K$ is defined over $\F_q$). As for $H_c^i(X/_\K,\kappa)$, we say that $IH_c^i(X/_\K,\kappa)$ is pure if the weight filtration is trivial.
 
\subsection{Quiver varieties}\label{quiver}

Here we assume that the characteristic of $\K$ is either $0$ or  sufficiently large to be able to make generic choices of eigenvalues of adjoint orbits (the bound can be made explicit, see for instance \cite[Section 2.2]{aha}). 
Fix a non-zero vector dimension $\v\in(\N)^r$ and  consider the moment map $\mu_\v:{\rm Rep}_\K(\overline{\Gamma},\v)\rightarrow\gl_\v^0$, $(x_\gamma)_{\gamma\in\overline{\Omega}}\mapsto \sum_{\gamma\in\Omega}[x_\gamma,x_{\gamma^*}]$ where $\gl_\v^0:=\{X\in\gl_\v\,|\,{\rm Tr}(X):=\sum_i{\rm Tr}(X_i)=0\}$. 

Since $\v$ is indivisible and since the characteristic of $\K$ is either $0$ or sufficiently large it is possible to choose $\xihat=(\xi_i)_i\in\K^r$ such that $\xihat\cdot\v:=\sum_i\xi_iv_i=0$ and $\xihat\cdot\w\neq 0$ for all $\w\in(\N)^r$ satisfying $\w\neq\v$ and $0\leq w_i\leq v_i$ for all $i$ (see \cite[Section 4.1]{letellier2} for more details). Call such a choice of $\xihat$ a \emph{generic} choice.

Choose now a nilpotent orbit $\calO$ of $\gl_\v$ and consider the $\GL_\v$-orbit $C=\xihat+\calO$ of $\gl_\v$, where by notation abuse we still denote by $\xihat$ the central element of $\gl_\v$ whose $i$-th coordinate is the central matrix with eigenvalue $\xi_i$. The orbit $C$ satisfies the following property.

\begin{lemma} We have ${\rm Tr}(C)=0$ and the only graded subspaces $V\subset\K^\v$ stable by some $X\in C$ such that 

$$
{\rm Tr}(X|_V)=0
$$
are either $0$ or $\K^\v$.

\label{gen}\end{lemma}

Define 

$$
\calQ_{\overline{C}}:=\mu_\v^{-1}(\overline{C})//\GL_\v={\rm Spec}\left(\K[\mu_\v^{-1}(\overline{C})]^{\GL_\v}\right).
$$
The action of $\GL_\v$ on $\mu_\v^{-1}(\overline{C})$ factorizes through an action of $G_\v:=\GL_\v/\bG_m$.

Denote by $C_i$ the component of $C$ in $\GL_{v_i}$ and denote by $\mu^i=(\mu^i_1,\mu^i_2,\dots,\mu^i_{s_i})$ the partition of $v_i$ given by the size of the Jordan blocks of $C_i$.

We now defined a new quiver $\Gamma_C$ together with a dimension vector $\v_C$ of $\Gamma_C$ as follows. At each vertex $i\in I$ of $\Gamma$ add a new leg of length equals to the length of $\mu^i$ minus $1$. For instance, if $\mu^i$ is of length $1$ i.e. $\mu^i=(v_i)$, then we do not add a leg and if $\mu^i$ is of maximal length, i.e., $\mu^i=(1^{v_i})$, we attach a leg of length $v_i-1$. We denote by $\Gamma_C$ the new quiver obtained from $\Gamma$ by adjoining these $r$ legs  and  where the arrows on the legs are oriented toward the vertices of $\Gamma$. Let $I_C$ denotes the set of vertices of $\Gamma_C$. We define $\v_C\in(\N)^{I_C}$ by putting the dimension vector $(v_i,v_i-\mu^i_1,v_i-\mu^i_1-\mu^i_2,\dots,\mu^i_{s_i})$ on the $i$-th leg. 

Define $\calQ_C$ as the image of $\mu^{-1}_\v(C)$ in $\calQ_{\overline{C}}$.

We have the following theorem.

\begin{theorem} The variety $\calQ_{\overline{C}}$ is not empty if and only if $\v_C$ is a root of $\Gamma_C$. Moreoever it is reduced to a single $\GL_\v$-orbits if and only if $\v_C$ is a real root.

Assume that $\calQ_{\overline{C}}$ is not empty, then

\noindent (i) the quotient map $\mu^{-1}_\v(\overline{C})\rightarrow\calQ_{\overline{C}}$ is a $G_\v$-principal bundle for the \'etale topology. In particular the $G_\v$-orbits of $\mu^{-1}_\v(\overline{C})$ are all closed of same dimension ${\rm dim}\,G_\v$.

\noindent (ii) $\calQ_C$ is a dense nonsingular open subset of $\calQ_{\overline{C}}$ (in particular it is non-empty),

\noindent (iii) the variety $\calQ_{\overline{C}}$ is irreducible of dimension

$$
2-{^t}\v_C C_{\Gamma_C}\v_C={\rm dim}\,{\rm Rep}_\K(\overline{\Gamma},\v)+{\rm dim}\, C-2\,{\rm dim}\, G_\v,
$$
where $C_{\Gamma_C}$ is the Cartan matrix of $\Gamma_C$.
\label{geoquiv}\end{theorem}

It is possible to identify  $\calQ_{\overline{C}}$ with a quiver variety $\mu^{-1}_{\v_C}(\xihat_C)//\GL_{\v_C}$ where $\mu_{\v_C}$ is the moment map for $\Gamma_C$ and where $\xihat_C\in \K^{I_C}$ is defined on each leg exactly as in \cite[page 1414]{letellier2}. We now prove this theorem using this identification together with the general theorems by Crawley-Boevey and Nakajima on quiver varieties of the form $\mu^{-1}_\v(\xihat)//\GL_\v$. See \cite[Section 4]{letellier2} and the reference therein for a summary of these general properties. 
\bigskip

Let us now construct a resolution of the variety $\calQ_{\overline{C}}$. 

For each $i\in I$, denote by $L_i$ the Levi subgroup $\prod_{j=1}^{p_i}\GL_{m_j}$ where $(m_1,m_2,\dots,m_{p_i})$ is the dual partition of $\mu^i$ and let $P_i$ be a parabolic subgroup of $\GL_{v_i}$ having $L_i$ as a Levi factor. Then $\bfP=P_1\times\cdots\times P_r$ is a parabolic subgroup of $\GL_\v=\GL_{v_1}\times\cdots\times\GL_{v_r}$ with Levi factor $\bfL=L_1\times\cdots\times L_r$. Denote by $\frak{u}_\bfP$ the Lie algebra of the unipotent radical of $\bfP$.

Define 

$$
\bV_{\bfL,\bfP,\xihat}:=\left\{(\rho,g\bfP)\in {\rm Rep}_\K(\overline{\Gamma},\v)\times (\GL_\v/\bfP)\,\left|\, g^{-1}\mu_v(\rho)g\in\xihat+\frak{u}_\bfP\right\}\right..
$$

Note that the diagonal action of $\GL_\v$ on ${\rm Rep}_\K(\overline{\Gamma},\v)\times (\GL_\v/\bfP)$ preserves the relation and so $\GL_\v$ acts on $\bV_{\bfL,\bfP,\xihat}$. The projection on the first coordinate $\rho:\bV_{\bfL,\bfP,\xihat}\rightarrow\mu^{-1}_\v(\overline{C})$ is then a $\GL_\v$-equivariant projective map.

Say that an adjoint orbit of $\gl_\v$ is \emph{generic} if it satisfies Lemma \ref{gen} and choose a generic semisimple orbit $S$ of $\gl_\v$ such that the multiplicities of the eigenvalues of the $i$-th component are given by the parts of the dual partition of $\mu^i$ (since $\v$ is indivisible such an adjoint orbit $S$ always exist \cite[Section 2.2]{aha}). We denote then by $\calQ_S$ the quiver variety $\mu_\v^{-1}(S)//\GL_\v$. When $S$ is regular, then $\calQ_S$ was considered and denoted $\calQ_\tv$ in \cite[Section 2]{aha3}. The quiver varieties $\calQ_S$ are non-singular and irreducible of same dimension as $\calQ_{\overline{C}}$.

We have the following theorem.

\begin{theorem} (i) The geometric quotient $\bV_{\bfL,\bfP,\xihat}\rightarrow\Q_{\bfL,\bfP,\xihat}$ exists and is a principal $G_\v$-bundle for the \'etale topology. Moreover the diagram

$$\xymatrix{
\bV_{\bfL,\bfP,\xihat}\ar[rr]^{\rho}\ar[d]&&\mu_\v^{-1}(\overline{C})\ar[d]\\
\Q_{\bfL,\bfP,\xihat}\ar[rr]^{\rho/_{G_\v}}&&\calQ_{\overline{C}}}
$$
where the vertical arrows are the quotient maps is Cartesian.

\noindent (ii) The varieties $\bV_{\bfL,\bfP,\xihat}$ and $\Q_{\bfL,\bfP,\xihat}$ are both irreducible non-singular and the map $\rho/_{G_\v}$ is a resolution of $\calQ_{\overline{C}}$.

\noindent (iii) The quiver varieties $\Q_{\bfL,\bfP,\xihat}$ and $\calQ_S$ have isomorphic compactly supported cohomology. Moreover their cohomology is pure and vanishes in odd degree.

\noindent (iv) The compactly supported intersection cohomology $IH_c^i(\calQ_{\overline{C}},\kappa)$ is pure and vanishes for odd $i$.
\label{theogeo2}\end{theorem}

This is an easy generalization of the case where the quiver $\Gamma$ consists of one vertex and $g$ loops which case is treated   in \cite[Sections 5.3, 7.3]{letellier2}. Note however that (iv) is a consequence of (ii) and (iii). Indeed, from (ii) we get that $IH_c^i(\calQ_{\overline{C}},\kappa)$ is a direct summand of $H_c^i(\Q_{\bfL,\bfP,\xihat},\kappa)$ as a mixed structure, hence the purity for $\calQ_{\overline{C}}$ follows from that of $\Q_{\bfL,\bfP,\xihat}$.
\bigskip

Now write $C_\muhat$ instead of $C$ and denote by $1^\v\in\calP_\v$ the multi-partition $\muhat$ with $\mu^i=(1^{v_i})$ for all $i=1,\dots,r$, and by $\v^1\in\calP_\v$ the multi-partition $\muhat$ with $\mu^i=(v_i)$ for all $i$. Then $C_{\v^1}=\xihat+C'$ where $C'$ is the nilpotent regular orbit of $\gl_\v$ and $C_{1^\v}$ consists of the single element $\xihat$. Denote by $\bfT$ the maximal torus of elements of $\GL_\v$ whose coordinates are diagonal matrices and denote by $\bfB$ the Borel subgroup of $\GL_\v$ consisting of upper triangular matrices. Consider the inclusion $i:\mu_\v^{-1}(\overline{C}_{\v^1})\rightarrow{\rm Rep}_\C(\overline{\Gamma},\v)\times\overline{C}_{\v^1}$, $\rho\mapsto (\rho,\mu_\v(\rho))$, and the Springer resolution 

$$
{\rm Sp}:\mathbb{X}_{\bfT,\bfB,\xihat}:=\{(x,g\bfB)\in\GL_\v\times(\GL_\v/\bfB)\,|\, g^{-1}xg\in\xihat+\frak{u}_\bfB\}\rightarrow\overline{C}_{\v^1},
$$
defined as  $(x,g\bfB)\mapsto x$.

We have the following Cartesian diagram

$$
\xymatrix{{\rm Rep}_\K(\overline{\Gamma},\v)\times\mathbb{X}_{\bfT,\bfB,\xihat}\ar[rr]^{{\rm Id}\times {\rm Sp}}&&{\rm Rep}_\K(\overline{\Gamma},\v)\times\overline{C}_{\v^1}\\
\bV_{\bfT,\bfB,\xihat}\ar[d]\ar[u]\ar[rr]^{\rho}&&\mu_\v^{-1}(\overline{C}_{\v^1})\ar[d]\ar[u]^i\\
\Q_{\bfT,\bfB,\xihat}\ar[rr]^{\rho/_{G_\v}}&&\calQ_{\overline{C}_{\v^1}}}
$$

We denote by $\pIC X$ the simple perverse sheaf on $X$ obtained by shifting by ${\rm dim} \,X$ the complex $\IC X$. If $X$ is non-singular, we denote by $\underline{\kappa}$ the constant perverse sheaf with $\kappa$ in degree $-{\rm dim}\,X$ and $0$ in other degrees.

Put $\frak{G}_\v=\prod_{i=1}^r\frak{G}_{v_i}$. Then $\frak{G}_\v$ is naturally identified with the Weyl group $N_{\GL_\v}(\bfT)/\bfT$ of $\GL_\v$. We parametrize the irreducible characters of $\frak{G}_\v$ by $\calP_\v$ in such a way that the identity character corresponds to $\v^1$. Now it follows from Springer's theory (see \cite[Section 6.4]{letellier2} for a review) that

$$
({\rm Id}\times {\rm Sp})_*(\underline{\kappa})\simeq \bigoplus_{\muhat\in\calP_\v}A_\muhat\otimes\pIC {{\rm Rep}_\K(\overline{\Gamma},\v)\times\overline{C}_\muhat},
$$
where $A_\muhat$ is an irreducible representation of $\frak{G}_\v$ corresponding to $\muhat$.
We can prove the following proposition as in \cite[Section 5.4]{letellier2}.

\begin{proposition}The restriction (pull back) of $\IC {{\rm Rep}_\K(\overline{\Gamma},\v)\times\overline{C}_\muhat}$ to $\mu_\v^{-1}(\overline{C}_\muhat)$ along the map $i$  is $\IC {\mu_\v^{-1}(\overline{C}_\muhat)}$.
\label{restriction}\end{proposition}

Moroever the codimension of $\mu_\v^{-1}(C_\muhat)$ in $\mu_\v^{-1}(\overline{C}_{\v^1})$ equals the codimension of $C_\muhat$ in $\overline{C}_{\v^1}$. Therefore we can use the above Cartesian diagram to see that 

$$
(\rho/_{G_\v})_*(\underline{\kappa})\simeq \bigoplus_{\muhat\in\calP_\v}A_\muhat\otimes\pIC {\calQ_{\overline{C}_\muhat}}.
$$
Taking hypercohomology we find an isomorphism

$$
H_c^i(\Q_{\bfT,\bfB,\xihat},\kappa)\simeq\bigoplus_{\muhat\in\calP_\v}A_\muhat\otimes IH_c^{i-\delta_\muhat}(\calQ_{\overline{C}_\muhat},\kappa),
$$
where $\delta_\muhat$ denotes the codimension of $\calQ_{\overline{C}_\muhat}$ in $\calQ_{\overline{C}_{\v^1}}$. The group $\frak{G}_\v$ acts on each $H_c^i(\Q_{\bfT,\bfB,\xihat},\kappa)$ so that for each $\w\in\frak{G}_\v$ we have 

\beq
{\rm Tr}\,\left(\w, H_c^i(\Q_{\bfT,\bfB,\xihat},\kappa)\right)=\sum_{\muhat\in\calP_\v}{\rm Tr}\,(\w,A_\muhat)\,\,{\rm dim}\, IH_c^{i-\delta_\muhat}(\calQ_{\overline{C}_\muhat},\kappa).
\label{tracefor}\eeq
For an equidimensional complex  algebraic variety $X/_\K$ with vanishing odd intersection cohomology we put

$$
P_c(X;t):=\sum_i{\rm dim}\, IH_c^{2i}(X,\kappa)\, t^i.
$$

Denote by $\rho^i$ the character of the representation of $\frak{G}_\v$ on $H_c^i(\Q_{\bfT,\bfB,\xihat},\kappa)$ defined above and by $\chi^\muhat$ the character of $A_\muhat$. From Formula (\ref{tracefor}) we deduce the following one.

\begin{proposition}

$$
P_c(\calQ_{\overline{C}_\muhat};t)=t^{-\frac{1}{2}\delta\muhat}\sum_i\langle\rho^{2i},\chi^\muhat\rangle_{\frak{G}_\v}t^i.
$$
\label{graded}\end{proposition}

\subsection{Fourier transforms, unipotent characters}

\subsubsection{Unipotent characters}\label{unipotent}

To alleviate the notation, put $G:=\GL_n(\F_q)$ and let $B\subset G$ be the Borel subgroup of upper triangular matrices and let $\kappa[G/B]$ be the $\kappa$-vector space with basis $G/B=\{gB\,|\,g\in G\}$. The group $G$ acts on $\kappa[G/B]$ by left multiplication. Let us denote by ${\rm Ind}_B^G(1):G\rightarrow\kappa, g\mapsto {\rm Trace}\,\left(g\,|\,\kappa[G/B]\right)$ the character of the representation of $G$ on $\kappa[G/B]$. The decomposition of ${\rm Ind}_B^G(1)$ as a sum of irreducible characters of $G$ reads

$$
{\rm Ind}_B^G(1)=\sum_{\chi\in {\rm Irr}\,\frak{G}_n}\chi(1)\cdot \calU_\chi.
$$
The irreducible characters $\{\calU_\chi\}_\chi$ are called the \emph{unipotent} characters of $G$. The character $\calU_1$ is the trivial character of $G$ and $\calU_\epsilon$, where $\epsilon$ is the sign character of $\frak{G}_n$, is the Steinberg character of $G$. For a partition $\mu$ of $n$, we put 

$$
\calU_\mu:=\calU_{\chi^\mu}
$$
so that the $\calU_{(1^n)}$ is the Steinberg character and $\calU_{(n^1)}$ is the trivial character.

The unipotent characters of $\GL_\v(\F_q)$, with  $\muhat=(\mu^1,\dots,\mu^r)\in\calP_\v$, are then $\calU_\muhat:=\calU_{\mu^1}\boxtimes\cdots\boxtimes\calU_{\mu^r}.$

\subsubsection{Fourier transforms of nilpotent orbits}\label{fourier}

In this section $\K=\overline{\F}_q$.

For a multi-partition $\muhat\in\calP_\v$, we denote by $\calO_\muhat$ the corresponding nilpotent orbit of $\gl_\v$. Define the characteristic function $\bfX_{\overline{\calO}_\muhat}:\gl_\v(\F_q)\rightarrow\Z$ of $\IC {\overline{\calO}_\muhat}$ as

$$
\bfX_{\overline{\calO}_\muhat}(x)=\sum_i{\rm dim}\,\calH_x^{2i}(\IC {\overline{\calO}_\muhat})\, q^i,
$$
(recall that $\IC {\overline{\calO}_\muhat}$ has vanishing odd cohomology). Since $\calO_{\v^1}$ is regular, its Zariski closure is rationally smooth and so  $\bfX_{\IC {\overline{\calO}_{\v^1}}}$ is the  function $\eta_o:\gl_\v(\F_q)\rightarrow\Z$ that takes the value $1$ on nilpotent elements and $0$ elsewhere.

Fix a non-trivial additive character $\psi:\F_q\rightarrow\kappa^\times$, and for $x,y$ in $\gl_\v$, put $\langle x,y\rangle:={\rm Tr}\,(xy)$. Note that the form $\langle\,,\,\rangle$ is non-degenerate. Denote by $\calC(\gl_\v)$ the $\kappa$-vector space of all functions $\gl_\v(\F_q)\rightarrow\kappa$, and define the Fourier transform $\calF:\calC(\gl_\v)\rightarrow\calC(\gl_\v)$ as 

$$
\calF(f)(x)=\sum_{y\in\gl_\v(\F_q)}\psi({\rm Tr}(xy))f(y),
$$
for $f\in\calC(\gl_\v)$ and $x\in\gl_\v(\F_q)$.

The values of the two class functions $\calU_\muhat$ and $\calF(\bfX_{\overline{\calO}_\muhat})$ depend only on the types of the orbits and we have the precise relationship between the two \cite[Formula (6.6.1), Theorem 6.7.1]{letellier2}.

\begin{theorem} Let $C\subset\GL_\v(\F_q)$ and $\calO\subset\gl_\v(\F_q)$ be any $\GL_\v(\F_q)$-orbits of same type. Then 

$$
\calU_\muhat(C)=q^{-\frac{1}{2}{\rm dim}\,\calO_{\muhat'}}\calF(\bfX_{\overline{\calO}_{\muhat'}})(\calO),
$$
where $\muhat'$ denotes the dual multi-partition of $\muhat$.
\label{Fourier-uni}\end{theorem}

\subsection{Cohomological interpretation of multiplicities}

Consider the generic adjoint orbits $C_\muhat$ over $\K$ as defined in \S \ref{quiver} ($\K$ can be either $\C$ or an algebraic closure of a finite field). Recall (see introduction or proof of Proposition \ref{multieq} below) that for each $\muhat\in\calP_\v$, there exists a polynomial $R_\v^\muhat(t)\in\Q[t]$ such that for any finite field $\F_q$ and any linear character $\alpha_\v:\mathbb{G}_m(\F_q)\rightarrow\kappa^\times$ of order $\sum_iv_i$, the evaluation $R_\v^\muhat(q)$ equals the multiplicity $\langle R_\v,1^{\alpha_\v}\otimes\calU_\muhat\rangle$ of the irreducible character $1^{\alpha_\v}\otimes\calU_\muhat$ of $\GL_\v(\F_q)$ in the big representation $R_\v$.

 We prove the following theorem.

\begin{theorem} For any $\muhat\in\calP_\v$ we have

$$
R_\v^\muhat(t)=t^{-\frac{1}{2}{\rm dim}\,\calQ_{\overline{C}_{\muhat'}}}P_c(\calQ_{\overline{C}_{\muhat'}}/_\K,t).
$$

\label{maintheo}\end{theorem}

If $\Gamma$ consists of one vertex with $g$ loops, then the theorem is exactly \cite[Theorem 7.3.3]{letellier2} with $k=1$ and $\calX_1=1^{\alpha_\v}\otimes\calU_\muhat$. The proof for any $\Gamma$ goes exactly along the same lines. Since we are working with unipotent characters, the proof simplifies and we sketch part of it for the convenience of the reader.

We need to prove an intermediate result. Given an element $X\in\gl_\v(\F_q)$ and $\rho\in{\rm Rep}_{\F_q}(\Gamma,\v)$, define $[X,\rho]\in {\rm Rep}_{\F_q}(\Gamma,\v)$ by 

$$
[X,\rho]_{i\rightarrow j}:=X_j\rho_{i\rightarrow j}-\rho_{i\rightarrow j}X_i,
$$
for any $i\rightarrow j\in\Omega$ and consider the function $\Theta_\v:\gl_\v(\F_q)\rightarrow\kappa$ given by 

$$
\Theta_\v(X):=\#\left\{\rho\in{\rm Rep}_{\F_q}(\Gamma,\v)\,\left|\,[X,\rho]=0\right.\right\}.
$$
Assume given a generic element $\xihat=(\xi_1,\dots,\xi_r)\in (\F_q)^r$ and denote by $1_\xihat:\gl_\v(\F_q)\rightarrow\kappa$ the characteristic function that takes the value one at $\xihat\in\gl_\v(\F_q)$ and the value $0$ elsewhere. Let $1^{\xihat_\v}$ denote the linear character $\calF(1_\xihat)$ of the abelian group $\gl_\v(\F_q)$. For two class functions $h_1,h_2$ on $\gl_\v(\F_q)$, define 

$$
\langle h_1,h_2\rangle_{\gl_\v}:=\frac{1}{|\GL_\v(\F_q)|}\sum_{x\in\gl_\v(\F_q)}h_1(x)\overline{h_2(x)}.
$$
To alleviate the notation put 

$$
\mathcal{N}_\muhat:=\calF(\bfX_{\IC {\overline{\calO}_{\muhat}}}).
$$

\begin{proposition} We have the following identity:

$$
\left\langle R_\v,1^{\alpha_\v}\otimes\calU_\muhat\right\rangle_{\GL_\v}=(1-q^{-1})\,q^{-\frac{1}{2}{\rm dim}\, C_{\muhat'}}\left\langle \Theta_\v, 1^{\xihat_\v}\otimes\mathcal{N}_{\muhat'}\right\rangle_{\gl_\v}.
$$
\label{multieq}\end{proposition}

\begin{proof}Denote by $\calU_\muhat^\omhat(q)$ the value of $\calU_\muhat$ at an element of $\GL_\v(\F_q)$ of type $\omhat\in\T_\v$. Then 

$$
\langle R_\v,1^{\alpha_\v}\otimes\calU_\muhat\rangle_{\GL_\v}=\sum_{\omhat\in\T_\v}\frac{1}{Z_\omhat(q)}R_\omhat(q)\calU_\muhat^\omhat(q)\sum_C1^{\alpha_\v}(C),
$$
where the second sum is over the conjugacy classes of type $\omhat$. By Formula (\ref{om}), we deduce that

$$
\langle R_\v,1^{\alpha_\v}\otimes\calU_\muhat\rangle_{\GL_\v}=(q-1)\sum_{\omhat\in\T_\v}C_\omhat^o\,\frac{1}{Z_\omhat(q)}R_\omhat(q)\calU_\muhat^\omhat(q).
$$
Denote by $\Theta_\omhat(q)$ and $\mathcal{N}_\muhat^\omhat(q)$ the respective value of $\Theta_\v$ and $\mathcal{N}_\muhat$ at an element of $\gl_\v(\F_q)$ of type $\omhat\in\T_\v$.  

Then

$$
\left\langle \Theta_\v,1^{\xihat_\v}\otimes\mathcal{N}_{\muhat'}\right\rangle_{\gl_\v}=\sum_{\omhat\in\T_\v}\frac{1}{Z_\omhat(q)}\Theta_\omhat(q)\mathcal{N}_{\muhat'}^\omhat(q)\sum_\calO1^{\xihat_\v}(\calO),
$$
where the second sum is over the adjoint orbits of $\gl_\v(\F_q)$ of type $\omhat$. The proposition follows thus from Theorem \ref{Fourier-uni}, the fact that $\Theta_\omhat(q)=R_\omhat(q)$ and the additive analogue of Formula (\ref{om}) which reads

$$
\sum_\calO1^{\xihat_\v}(\calO)=qC_\omhat^o.
$$

\end{proof}

Denote by $\Theta_\v^\muhat(t)\in\Q[t]$ the polynomial whose evaluation at $q$ equals $\left\langle \Theta_\v, 1^{\xihat_\v}\otimes\mathcal{N}_\muhat\right\rangle_{\gl_\v}$. Theorem \ref{maintheo} is now a consequence of the following theorem.

\begin{theorem} We have

$$
\Theta_\v^\muhat(t)=(t-1)^{-1}t^{{\rm dim}\,\GL_\v-{\rm dim}{\rm Rep}_\K(\Gamma,\v)}P_c(\calQ_{\overline{C}_\muhat}/_\K,t).
$$

\label{Theta}\end{theorem}

\begin{proof}If $\Gamma$ consists of only one vertex with $g$ loops, then the theorem is exactly the formula displayed in \cite[Proof of Theorem 7.3.3]{letellier2} with $k=1$. The proof for an arbitrary $\Gamma$ is completely similar. However we will still give the complete proof for the case $\muhat=1^\v$ as it is much simpler than the general case and also this is the case we are interested for DT-invariants.  The fact that the proof in this case is much simpler comes from the fact that the variety $\calQ_{\overline{C}_{\v^1}}$ is rationally smooth. Recall that $C_{\v^1}=\xihat+\calO_{\v^1}$. To alleviate the notation put $C=C_{\v^1}$, $\calO=\calO_{\v^1}$ and $\mathcal{N}=\mathcal{N}_{\v^1}$. By Hausel's formula (see \cite[Proposition 2.9]{aha3}, we have

$$
|\calQ_{\overline{C}}(\F_q)|=\frac{(q-1)|{\rm Rep}_{\F_q}(\Gamma,\v)|}{|\GL_\v(\F_q)|\cdot |\gl_\v(\F_q)|}\sum_{X\in\gl_\v(\F_q)}\Theta_\v(X)\,\calF(1_{\overline{C}})(X),
$$
where $1_{\overline{C}}\in\calC(\gl_\v)$ is the characteristic function of $\overline{C}$, i.e., it takes the value $1$ on $\overline{C}$ and $0$ elsewhere.  But $C=\xihat+\calO$. Therefore, we have $1_{\overline{C}}=1_{\xihat}*\eta_o$ where $\eta_o$ is the characteristic functions of nilpotent elements and $*$ the convolution product on functions on $\gl_\v(\F_q)$.  We thus have $\calF(1_{\overline{C}})=1^{\xihat_\v}\otimes\mathcal{N}$ and so

$$
|\calQ_{\overline{C}}(\F_q)|=\frac{(q-1)|{\rm Rep}_{\F_q}(\Gamma,\v)|}{|\gl_\v(\F_q)|}\left\langle\Theta_\v,1^{\xihat_\v}\otimes\mathcal{N}\right\rangle_{\gl_\v}.
$$

Since $\calQ_{\overline{C}}$ has polynomial count (over a finite field) and since it is cohomologically pure we deduce that 

$$
\sum_i{\rm dim}\,H_c^{2i}(\calQ_{\overline{C}}/_\K,\kappa)\, q^i=|\calQ_{\overline{C}}/_{\overline{\F}_q}(\F_q)|.
$$
We refer to Katz appendix in \cite{hausel-villegas} for details relating the cohomology of polynomial count varieties with their counting polynomial (see \cite[Section 3.3]{letellier2} for generalisation of Katz theorem to intersection cohomology which is needed for general $\muhat$). The variety $\calQ_{\overline{C}}$ is rationally smooth and so we have $IH_c^{i}(\calQ_{\overline{C}}/_\K,\kappa)=H_c^i(\calQ_{\overline{C}}/_\K,\kappa)$ from which we deduce the theorem for $\muhat=\v^1$. The rational smoothness of $\calQ_{\overline{C}}$ is a consequence of the rational smoothness of the Zariski closures of regular orbits together with the fact that the restriction of the intersection cohomology complex on ${\rm Rep}_\K(\overline{\Gamma},\v)\times\overline{C}_{\muhat}$ to $\mu_\v^{-1}(\overline{C}_{\muhat})$ is again an intersection cohomology complex (see Proposition \ref{restriction}).
\end{proof}

Putting $\muhat=\v^1$ in Theorem \ref{maintheo}  we obtain by Proposition \ref{Kacprop}

$$
A_\v(t)=t^{-\frac{1}{2}{\rm dim}\,\calQ_\xihat}\sum_i{\rm dim}\, H_c^{2i}(\calQ_\xihat,\kappa)\, t^i,
$$
where  $\calQ_\xihat=\mu_\v^{-1}(\xihat)//G_\v$. This formula was first proved by Crawley-Boevey and van den Bergh \cite{crawley-etal} by a completely different method (note that $IH_c^i(\calQ_\xihat,\kappa)=H_c^i(\calQ_\xihat,\kappa)$ as $\calQ_\xihat$ is non-singular).

We deduce the main result which was one of the main motivation for this paper.

\begin{corollary}We have

$$
\DT_\v(t)=t^{-\frac{1}{2}{\rm dim}\,\calQ_{\overline{C}_{\v^1}}}\sum_i{\rm dim}\, H_c^{2i}(\calQ_{\overline{C}_{\v^1}},\kappa)\, t^i.
$$
\end{corollary}

\begin{proof}This follows from Theorem \ref{maintheo} and  Proposition \ref{DTprop} together with the fact that $\calQ_{\overline{C}_{\v^1}}$ is rationally smooth (see proof of Theorem \ref{Theta}).
\end{proof}
 
 By Theorem \ref{theogeo2} (iii), we have $H_c^i(\Q_{\bfT,\bfB,\xihat},\kappa)\simeq H_c^i(\calQ_S,\kappa)$, where $S$ is semisimple regular generic adjoint orbit of $\gl_\v$.  Therefore, by the discussion at the end of \S \ref{quiver}  we have an action of $\frak{G}_\v$ on $H_c^i(\calQ_S,\kappa)$, and we deduce from Theorem \ref{maintheo} together with Proposition \ref{graded} the following result.
 
 \begin{corollary} The polynomial $t^{\frac{1}{2}{\rm dim}\,\calQ_S}R_\v^\muhat(t)$ is the graded multiplicity of the irreducible character $\chi^{\muhat'}$ in the graded $\kappa[\frak{G}_\v]$-module $H_c^{2\bullet}(\calQ_S/_\K,\kappa)$.
\label{gradedres}\end{corollary}

   \end{document}